\documentclass[a4paper,11pt]{article}

\usepackage[utf8]{inputenc}
\usepackage[T1]{fontenc}
\usepackage[english]{babel}
\usepackage{amsmath,amssymb,amsthm}
\usepackage{mathtools}
\usepackage{hyperref}
\usepackage{xcolor}
\usepackage{authblk}
\usepackage{caption}
\usepackage{gensymb}
\usepackage{bbm}
\usepackage{enumitem}
\usepackage{geometry}
\usepackage{graphicx}
\usepackage[
    style=numeric, 
    backend=biber,
    doi=true,
    isbn=false,
    url=true
]{biblatex}
\theoremstyle{definition} 
\newtheorem{theorem}{Theorem}[section]
\newtheorem{lemma}[theorem]{Lemma}
\newtheorem{proposition}[theorem]{Proposition}
\newtheorem{definition}[theorem]{Definition}
\newtheorem{remark}[theorem]{Remark}
\newtheorem{note}[theorem]{Note}
\newtheorem{claim}[theorem]{Claim}
\newtheorem{corollary}[theorem]{Corollary}
\newtheorem{example}[theorem]{Example}

\newcommand{\R}{\mathbb{R}}
\newcommand{\Haus}{\mathcal{H}}

\title{The lens cluster and triod cluster uniquely minimize the anisotropic perimeter in $\R^2$}
\author{Paula Benítez Sesmilo}
\affil{CFIS, UPC}
\date{}
\begin{document}
\maketitle

\begin{abstract}
\noindent $(N, M)$-clusters are partitions of $\R^d$ into $N+M$ regions, where $N$ chambers have prescribed finite measure and $M$ chambers have infinite measure. \emph{Locally minimizing clusters} are the configurations which minimize the perimeter among all competitors with compact support satisfying the same measure constraints.
The characterization of these partitions has been widely studied for the standard (isotropic) perimeter. In the present paper, we investigate the corresponding problem for anisotropic perimeters, considering a general anisotropy. More specifically, we focus on $(1,2)$-clusters and $(1,3)$-clusters in $\R^2$. Our main results provide a geometric characterization of these local minimizers: for regular (smooth, symmetric, and uniformly convex) anisotropies, we prove that a cluster is a local minimizer if and only if, up to translations, it is a \emph{standard anisotropic lens cluster} in the $(1,2)$-cluster case, or a \emph{standard anisotropic triod cluster} in the $(1,3)$-cluster case. In addition, using an approximation argument, we extend the minimizing property of these configurations to general anisotropies. 
\end{abstract}

\tableofcontents

\section{Introduction}
Classical isoperimetric problems consist of finding, among all sets of given measure (generally Lebesgue measure), those whose boundary has the smallest possible perimeter; i.e., the set $E \subset \mathbb{R}^d$ that minimizes the perimeter $P(E)$. Cluster problems arise as a natural extension of the classical isoperimetric problem. One studies collections of disjoint measurable sets, called chambers, with prescribed (finite or infinite) measures. These configurations are commonly named $(N,M)$-clusters, where $N$ denotes the number of chambers with finite measure and $M$ the number of chambers with infinite measure. Their geometry aims to minimize the total surface area between the chambers. Since interfaces shared by two chambers are counted only once, optimal configurations usually exhibit nontrivial geometric structures, giving rise to the mathematical theory of clusters. The existence of such configurations is guaranteed by the direct method of the calculus of variations \cite{Almgren1976}.

Cluster problems have been extensively studied in the isotropic case. For the most basic case, $N=1$, it is established that the solution which minimizes the perimeter for a prescribed volume is a ball.
In the planar case ($d=2$),  minimizers of $(N, 1)$-clusters admit a rigorous characterization \cite{Morgan1994}. In this framework, minimizers are composed of straight segments and circular arcs meeting at triple junctions, where the angles between interfaces are equal to $120^\circ$, often referred to as the Steiner rule. In the case of multiple regions, $N\geq 1$, the natural cluster is called the \emph{standard N-bubble}, and it is conjectured to be the unique minimizer, up to isometries. For the case $N=2$, the unique minimizer is the standard double bubble, proved for $d=2$ by Foisy et al. \cite{Foisy1993} and for $d \geq 2$ by Hutchings-Morgan-Ritoré-Ros \cite{Hutchings2004}. Regarding larger clusters, results have been obtained for the triple-bubble problem ($N=3$) in the plane by Wichiramala \cite{Wichiramala2004}, and have been established for $d \geq 2$ by Milman and Neeman \cite{Milman2022}. The same authors proved the isoperimetric property of standard clusters for $N=4$ ($d \geq 4$) \cite{Milman2022} and $N=5$ ($d \geq 5$) \cite{Milman2023}. 

More recently, attention has also turned to cluster problems involving more than one infinite measure chamber. In this case, the total perimeter of the configuration is necessarily infinite, and global minimization must be replaced by a local variational notion, as introduced by Alama, Bronsard, and Vriend in \cite{Alama2023}. In that work, they studied planar partitions consisting of one bounded region and two unbounded ones, corresponding to a $(1,2)$-cluster. They proved the existence and uniqueness of the so-called \emph{standard lens cluster}, whose boundary consists of two symmetric circular arcs joined to two half-lines at triple junctions satisfying the Steiner rule. For their proof, they relied on the existence of the double bubble cluster. Novaga, Paolini and Tortorelli generalized the properties of this type of partitions, proving a \emph{closure theorem} for local minimizers when the limiting cluster has flat interfaces outside a compact set \cite{Novaga2023}. In the same paper, they proved that for locally minimizing clusters in the plane the maximum number of infinite chambers for a local minimizer is $3$. They also characterized $(N, M)$-cluster minimizers for $N+M\leq 4$: specifically, for the $(2,2)$-cluster we have the \emph{peanut partition}, and for the $(1,3)$-cluster the \emph{Reuleaux partition}.

    \begin{figure}[htbp]
    \centering
    \includegraphics[width=0.75\textwidth]{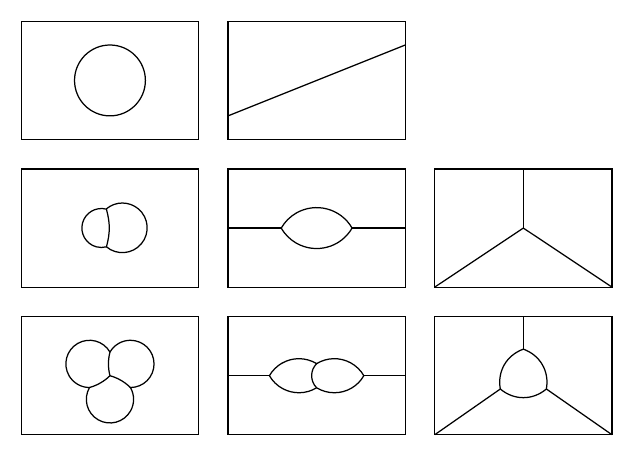} 
    \caption{Local minimizing clusters for the plane. Columns from left to right represent configurations with one, two, and three infinite chambers, respectively. The rows show an increasing number of finite chambers. }
    \label{fig:wulff_lens}
\end{figure}

Motivated by these results, it is natural to investigate how the results change when the Euclidean perimeter is replaced by an anisotropic surface energy. In this paper, we work with the anisotropic perimeter, which is given by an l.s.c. function $\phi : \mathbb{S}^{d-1} \rightarrow (0, + \infty)$, that extends to $\R^d$ as a norm (convex and positively 1-homogeneous). The anisotropic perimeter of any set $E \subseteq \mathbb{R}^d$ of locally finite perimeter is defined by
\[
P_\phi(E)= \int_{\partial^*E} \phi(\nu_E(x)) \, d\mathcal{H}^{d-1}(x).
\]
where $\partial ^* E$ is the reduced boundary of $E$, i.e., for every $x \in \partial ^* E $ there exists an outer normal vector to $E$ at $x$, $\nu_E(x)$.
The anisotropy breaks the rotational invariance of the problem and modifies the geometry of the interfaces and the equilibrium conditions at the junctions, as previously discussed by Franceschi, Pratelli, and Stefani \cite{Franceschi2023}. The authors prove that, if the anisotropy is $\mathcal{C}^1$ and uniformly convex, and the studied cluster satisfies suitable growth and volume-adjustment assumptions, then the local minimizing cluster is a locally finite union of $\mathcal{C}^1$ arcs, and each junction point is the endpoint of exactly three different arcs, arriving with three different tangent vectors \cite[Theorem A]{Franceschi2023}. In this setting, the tangent vectors at triple points satisfy the anisotropic Young's law \cite[Theorem 3.6]{Franceschi2023}. 

Essential to the anisotropic isoperimetric problem is the Wulff shape, which is the unique (up to translations) minimizer of the anisotropic perimeter under a volume constraint \cite{Maggi2012}. In this sense, the Wulff shape generalizes the Euclidean ball and provides the natural building block for anisotropic minimal surfaces.

We study the minimization of the anisotropic perimeter of a $(1, 2)$-cluster (one finite region and two infinite regions) and a $(1, 3)$-cluster (one finite region and three infinite regions) in $\R^2$. In analogy with the isotropic case, one therefore expects the anisotropic cluster minimizers to be constructed from arcs of the Wulff shape and straight segments, arranged so as to satisfy suitable anisotropic balance conditions at junctions. In this paper, we prove that for a symmetric and regular anisotropy, a cluster is a local minimizer if and only if, up to translations, it is given by the \emph{standard anisotropic lens cluster} for the anisotropic isoperimetric problem involving one region of finite volume $m$ and two regions of infinite volume (Theorem~\ref{th:ifonlyif_lens}). Moreover, the same result is obtained for the analogous problem involving one region of finite volume $m$, and three regions of infinite volume, where the optimal configuration is the \emph{standard anisotropic triod cluster} (Theorem~\ref{th:ifonlyif_triod}). Furthermore, we obtain the minimizing property of these configurations for the general anisotropy case,  by using an approximation argument (Corollary~\ref{cor:mainresult_cor_lens} and Corollary~\ref{cor:mainresult_cor_triod}).\\

In Section~\ref{section:notation}, we establish important definitions for our study, as well as some properties of the Wulff shape. In Section~\ref{section:regularity_th}, we collect some important results that we will use to prove the existence and uniqueness of our minimizers. Using Theorem~\ref{theorem:boundary_cond}, we obtain a characterization of the asymptotic behavior of our partitions, allowing us to restrict the problem to a ball $B_R$. Then, in Section~\ref{section:local_pb}, we characterize the geometry of the minimizer in the ball and deduce its uniqueness for the regular anisotropy case. Moreover, we prove it preserves the minimizing property for general anisotropies. In Section~\ref{section:main}, we use Theorem~\ref{theorem:closure}, a closure theorem, to extend both results by passing to the limit as $R\rightarrow\infty$, obtaining the minimizers in $\R^2$.

\section{Notation and preliminary results} \label{section:notation}
A measurable set $E \subset \R^d$ is said to be a \emph{Caccioppoli set} if the distributional derivative $D\chi_E$ of its characteristic function is a Radon vector measure. 
For any open set $B \subset \R^d$, we define the standard perimeter of $E$ in $B$ by
\[
P(E,B) \coloneq \|D\chi_E\|(B),
\]
and we set $P(E) \coloneq P(E,\R^d)$.

Throughout the paper, sets are identified up to Lebesgue negligible sets. 
In particular, we choose representatives such that the measure theoretic boundary
coincides with the topological boundary. 
The reduced boundary $\partial^*E \subset \partial E$ comprises the points where the approximate outer normal unit vector $\nu_E$ can be defined so that  $D\chi_E = \nu_E \cdot \|D\chi_E\|$. If $E$ is a Caccioppoli set then
\[
P(E,B) = \Haus^{d-1}(\partial^*E \cap B)
\qquad \forall B \subset \R^d \text{ open set }.
\]

\subsection{Partitions and anisotropic perimeter} 
\begin{definition} [Partition]
Let $\mathcal{E} = (E_1, \dots,E_S)$ be an S-tuple of measurable subsets of $\R^d$, and $\Omega$ an open set. We say that $\mathcal{E}$ is an \emph{S-partition} (of $\Omega$), with \emph{regions} $E_1,\dots, E_S$, if $ |E_i \cap E_j \cap \Omega | = 0 $ $\forall i \neq j$, and $|\Omega \setminus \bigcup_i E_i | = 0$.

The \emph{boundary} $\partial \mathcal{E}$ of a partition is the set of all interfaces between the regions:
\[
\partial \mathcal{E} \coloneq \bigcup_{k=1}^S \partial E_k
\]

We define the \emph{perimeter} of $ \mathcal{E}$ on any open set $B\subset \R^d$ as:
\[
P(\mathcal{E}, B) \coloneq \frac{1}{2}\sum_{k=1}^S P(E_k, B), \qquad P(\mathcal{E})= P(\mathcal{E}, \R^d)
\]

\end{definition}

\begin{note}
    Notice that, unlike in bounded domains $\Omega$, when working in $\R^d$ the regions of a partition may have either finite or infinite Lebesgue measure:
    \[
        0 \le |E_i| \le +\infty .
    \]
    It is convenient to adopt the notation used in \cite{Alama2023}, dividing $\mathcal{E}$ into
    \[
     \mathcal{E} =(\mathcal{G}, \mathcal{F})= (G_1,\dots, G_N,\, F_1,\dots, F_M),
    \]
    where 
    \[
    |G_j| < \infty \quad 1 \leq j \leq N,
    \qquad
    |F_i| = \infty \quad 1 \leq i \leq M,
    \qquad
    N + M = S.
    \]
    This way one obtains a $(N, M)$-tuple, distinguishing \emph{proper} regions (finite volume) from \emph{improper} ones (infinite volume), a convention which will be useful when formulating local perimeter-minimization problems in $\R^d$.
    
\end{note}

\begin{definition}[Locally minimizing partition / locally isoperimetric partition]
Let $\mathcal{E}=(E_1,\dots,E_S)$ be a partition of $\Omega \subset \R^d$.
We say that $\mathcal{E}$ is \emph{locally minimizing} (or a \emph{locally isoperimetric
partition}) if for every compact set $B\subset\Omega$ and every competitor partition 
$\mathcal{F}$ such that

\[
 |E_i\Delta F_i \setminus B |=0,
\qquad
|F_i| = |E_i| \ \text{for all regions with finite volume},
\]

one has
\[
P(\mathcal{E};B) \le P(\mathcal{F};B).
\]

\end{definition}

\begin{definition}[Eventually flat partitions]
An $S$-partition $\mathcal{E}=(E_1, \dots, E_S)$ of $\R^d$ is \emph{eventually flat} if for every pair of indices $i \neq j$  such that $E_i$ and $E_j$ have infinite measure, there exists a $(d-1)$-dimensional half space contained in the interface $\partial E_i\cap \partial E_j$.
    
\end{definition}

\begin{definition}[Anisotropic perimeter of a set]
Let $E \subset \R^d$ be a set of finite perimeter, and let 
$\phi : \mathbb{S}^{d-1} \to (0,+\infty)$ be a continuous and positive anisotropic density.
The \emph{anisotropic perimeter} of $E$ in an open set $B \subset \R^d$ is defined by

\[
P_\phi(E;B)
:= \int_{\partial^* E \cap B} \phi(\nu_E(x)) \, d\mathcal{H}^{d-1}(x).
\]
The anisotropic perimeter of $E$ is $P_\phi(E)=P_\phi(E;\R^d)$.

We define the \emph{anisotropic perimeter} of a partition $\mathcal{E}=(E_1,\dots,E_S)$ on any open set $B \subset \R^d $ as

\[
P_\phi(\mathcal{E};B) = \frac{1 }{2} \sum_{i=1}^{S} P_\phi(E_i; B).
\]

In the planar case with a symmetric anisotropic density, we can express the perimeter of a partition in a much simpler way:
\[
P_\phi(\mathcal{E};B)
:= \int_{\partial^* \mathcal{E} \cap B} \phi(\nu_\mathcal{E}(x)) \, d\mathcal{H}^1(x),
\]
where $\partial^* \mathcal{E}$ is the reduced boundary of $\mathcal{E}$ and $\nu_\mathcal{E}$ is the measure–theoretic unit normal.

\end{definition}

\begin{note}
In the rest of the text, we assume that the extension of the anisotropic density $\phi : \mathbb{R}^2 \to [0,+\infty)$ is a positively 1-homogeneous convex function. We call an anisotropy satisfying these properties a \emph{general anisotropy}. Since the anisotropic density $\phi$ is continuous on a compact set, there exist constants $0 < \phi_{\min} \le \phi_{\max} < +\infty$ such that $
            \phi_{\min} \le \phi(\nu) \le \phi_{\max}$ for all $\nu \in \mathbb{S}^1$. Thus, its extension is locally bounded.
    
\end{note}

\begin{note}[Properties of regular anisotropies]\label{note:properties_anisotropy} 
We say an anisotropy is \emph{regular} if it is a general anisotropy satisfying the following properties:

    \begin{enumerate}[label=\roman*)]
    \item Regularity: 
        $\phi$ is of class $\mathcal{C}^2(\mathbb{R}^2 \setminus \{0\})$. 

    \item Uniform convexity:
$\phi$ is uniformly convex. This geometric condition means that the unit ball $\mathcal{W}_\phi = \{ \nu : \phi(\nu) \le 1 \}$ has strictly positive curvature.
\\

    \item Symmetry:
    $\phi$ is a symmetric function, i.e., $\phi(\nu)=\phi(-\nu)$ for all $ \nu \in \mathbb{S}^1 $.
\end{enumerate}

\end{note}

\begin{theorem}[Vol'pert] \label{theorem:volpert}
Let $E \subseteq \R^2$ be a set of locally finite perimeter, and let $x \in \R^2$ be fixed. Then, for a.e. $r > 0$, one has that
\[
\partial^* E \cap \partial B(x,r) = \partial^* (E \cap \partial B(x,r)) .
\]
\end{theorem}

 The theorem states that for almost every $r > 0$ the intersection $E \cap \partial B(x,r)$ consists basically of a finite union of arcs of the circle. Through the rest of the paper, we will often consider intersections of sets with balls; in such cases, we will always assume that Vol'pert's Theorem holds true.

\subsection{Wulff shape and initial results}

\begin{definition}[Wulff shape]
The Wulff shape associated with $\phi$ is
\[
W_\phi := \bigcap_{y\in \mathbb{S}^1} 
  \{x\in \R^2 : x\cdot y \le \phi(y)\}=\{\phi^* < 1 \}
\]
where $ \phi^*(z)=\sup\{z \cdot y : \phi(y)<1\} $.

It is the unique (up to translations) solution of the anisotropic version of the Euclidean isoperimetric problem
\[
\min \{ P_\phi(E) : |E| = m \}.
\] 
\end{definition}

\begin{lemma} \label{lemma:anisotrop_lens} 
    Let $\phi \in \mathcal{C}^2(\mathbb{R}^2 \setminus \{0\})$ be a regular anisotropy. Given a unit vector $\hat{n} \in \mathbb{S}^1 $ and a volume $m>0$, there exists a unique set $L_\phi\subset\R^2$ (up to translations) with the following properties:
    \begin{enumerate}
        \item $L_\phi$ is bounded by two arcs of a scaled Wulff shape $\lambda W_\phi$ for some $\lambda>0$.
        \item $\partial L_\phi$ contains exactly two vertices.
        \item At both junctions, the two outward normals from $L_\phi$, denoted by $\nu_1, \nu_2$, satisfy Young's law (the standard regularity condition at triple points \cite[Theorem 3.6]{Franceschi2023}) with respect to the exterior prescribed outer normal $\hat{n}$,
        \[
        \nabla \phi(\hat{n}) + \nabla \phi(\nu_1) + \nabla \phi(\nu_2) = 0.
        \]
        \item The enclosed area of $L_\phi$ equals $m$.
    \end{enumerate}

    We call this set the \emph{anisotropic lens} with prescribed outer normal $\hat{n}$ and volume $m$.

    \begin{proof}

The proof proceeds in several steps. First, we establish a one-to-one correspondence between $\mathbb{S}^1$ and the boundary of the Wulff shape $\partial W_\phi$. Then, for a prescribed normal $\hat n$, the anisotropic Young law uniquely determines the corresponding points on $\partial W_\phi$.  Hence, the two points where the Wulff arcs must meet are determined. These two points uniquely fix the shape of the lens, and the enclosed area $m$ then determines the scaling factor $\lambda$. 
\medskip

    \textit{Step 1: Diffeomorphism between $\mathbb{S}^1$ and $\partial W_\phi$.} 
We define the map
\[
F : \mathbb{S}^1 \longrightarrow \mathbb{R}^2,
\qquad 
F(\nu):= \nabla\phi(\nu).
\]
Since $\phi \in \mathcal{C}^2(\mathbb{R}^2 \setminus \{0\})$, its gradient is a $\mathcal{C}^1(\mathbb{R}^2 \setminus \{0\}) $ function, thus $F \in \mathcal{C}^1(\mathbb{S}^1)$. 
Since $\phi$ is $\mathcal{C}^2$ and strictly convex, its tangential Hessian is strictly positive, which implies that the derivative of $F$ in tangential directions never vanishes. Therefore $F$ is a $\mathcal{C}^1$-immersion. Strict convexity also implies that $\nabla \phi$ is injective on $\mathbb{S}^1$. Since  $\mathbb{S}^1$ is compact, $F$ is a homeomorphism onto its image, and combining this with the immersion property we conclude that $F$ is a $\mathcal{C}^1$-diffeomorphism between  $\mathbb{S}^1$ and $F(\mathbb{S}^1)$.

We now prove that $F(\mathbb S^1)=\partial W_\phi$. By Euler's theorem for
one-homogeneous functions,
\[
\phi(\nu)=\nabla\phi(\nu)\cdot\nu \qquad \text{for all } \nu\in\mathbb S^1.
\]
Combined with the convexity of $\phi$, this yields
\[
\nabla\phi(\nu)\cdot\eta \le \phi(\eta)
\qquad \text{for all } \eta\in\mathbb S^1.
\]
Hence, for any $x\in F(\mathbb S^1)$, there exists $\nu\in\mathbb S^1$ such that
$x=\nabla\phi(\nu)$, and the inequality above shows that $x\in W_\phi$.
Moreover, equality holds for $\eta=\nu$, so $x\in\partial W_\phi$. This proves
that $F(\mathbb S^1)\subset\partial W_\phi$.

 Conversely, if $x \in \partial W_\phi$, the functional $g(y)=x \cdot y-\phi(y)$  achieves its maximum at some $\nu \in \mathbb{S}^1$. Strict convexity of $\phi$ implies that this maximizer is unique, and the optimality condition $\nabla g(\nu)=0$ yields $x=\nabla\phi(\nu)$, thus  $x \in F(\mathbb{S}^1) $, proving $\partial W_\phi \subset F(\mathbb{S}^1)$. We conclude that
$F(\mathbb S^1)=\partial W_\phi$.

\medskip

\textit{Step 2: The points of the Wulff shape satisfying Young's law and its scaling are determined.}

Let $\hat n\in\mathbb S^1$ be the prescribed outer normal, and set
$A:=\nabla\phi(\hat n)\in\partial W_\phi$. We claim that there exists a unique unordered pair of points $\{B, C\}\subset\partial W_\phi\setminus\{A\}$ such that the three corresponding normals at each point satisfy the anisotropic Young's law
\[
\nabla\phi(\hat{n})+\nabla\phi(\nu _1)+\nabla\phi(\nu _2)=0,
\]
where $B=\nabla\phi(\nu _1)$ and $C=\nabla\phi(\nu _2)$ for specific $\nu_1, \nu_2 \in \mathbb{S}^1$.
The existence and uniqueness of such a pair  $\{B,C\}$ is proved in~\cite[Lemma 3.4]{Franceschi2023}. Specifically, they prove that for a symmetric anisotropic density $\phi$, given any fixed point $A\in \partial W_\phi$ there exists a unique triple of points $A, B, C \in \partial W_\phi$ such that the normal vectors at each point, namely $\hat{n}, \nu_1, \nu_2$ (using the defined diffeomorphism, we have the relation $B=\nabla\phi (\nu_1), \quad C=\nabla\phi (\nu_2) $), satisfy Young's law.
      
We now consider the arc of $\partial W_\phi$ connecting $B$ and $C$, namely $\Gamma$. Since the Wulff shape is centrally symmetric with respect to $O$ and strictly convex, the arc connecting
$B$ and $C$ is mapped to a second congruent arc $\Gamma^*$, connecting the reflected points $B'$ and $C'$. 
Let $\vec{d}$ be the vector connecting the endpoints. We translate the Wulff arc $\Gamma^*$ by $\vec{d}$ until the endpoints of both arcs coincide. The interface normals  $\nu_1$, $\nu_2$ at these points, together with the exterior prescribed normal $\hat{n}$, satisfy the anisotropic Young's law.

This configuration is uniquely determined up to translations, and we call the enclosed region the anisotropic lens.

Finally, scaling the Wulff shape by a factor $\lambda>0$ scales the lens
accordingly, and the enclosed area depends continuously and strictly
monotonically on $\lambda$. Therefore, there exists a unique $\lambda>0$ such
that the enclosed area equals the prescribed value $m$. This completes the
proof.
\end{proof}

\end{lemma}

    \begin{figure}[htbp]
    \centering
    \includegraphics[width=0.55\textwidth]{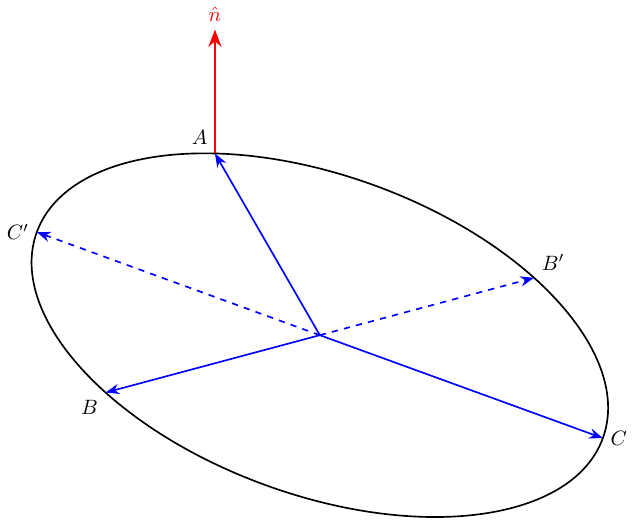} 
    \caption{Construction of the Wulff anisotropic lens.}
    \label{fig:wulff_lens}
\end{figure}

  \begin{lemma} \label{lemma:anisotrop_triple_lens} 
    Let $\phi \in \mathcal{C}^2(\mathbb{R}^2 \setminus \{0\})$ be a regular anisotropy. Given a unit vector $\hat{n} \in \mathbb{S}^1 $ and a volume $m>0$, there exists a unique set $T_\phi\subset\R^2$ (up to translations) with the following properties:
    \begin{enumerate}
        \item $T_\phi$ is bounded by three arcs of a scaled Wulff shape $\lambda W_\phi$ for some $\lambda>0$.
        \item $\partial T_\phi$ contains exactly three vertices.
        \item At the junctions, the two outward normals from $T_\phi$ satisfy Young's law - the standard regularity condition at triple points \cite[Theorem 3.6]{Franceschi2023} - with respect to another suitable direction. At one vertex, the law is satisfied with respect to the prescribed exterior normal $\hat{n}$,
        \[
        \nabla \phi(\hat{n}) + \nabla \phi(\nu_1) + \nabla \phi(\nu_2) = 0
        \]
        and at the remaining vertices, it is satisfied with respect to the normals $\nu_1$ and $\nu_2$, respectively.
    \end{enumerate}

    We call this set the \emph{anisotropic Reuleaux triangle} with prescribed outer normal $\hat{n}$ and volume $m$.

\begin{proof}
The proof is analogous to the previous one, with minor changes. 

Once the diffeomorphism is constructed, the points $\{A, B, C\}\subset \partial W_\phi$ are well-defined with the corresponding normals $\{\hat{n}, \nu_1, \nu_2\}$. We consider $A'$, $B'$ and $C'$ to be the reflections of $A,B,C$ with respect to
the center of symmetry of $W_\phi$. We then consider three arcs on $\partial W_\phi$: the arc connecting $A$ to $C'$ (namely $\Gamma_1$), the arc connecting $C$ to $B'$ ($\Gamma_2$), and the arc connecting $B$ to $A'$ ($\Gamma_3$).
We translate the arcs until they meet end-to-end, forming a closed curve that encloses a simply connected region. By construction, at the vertex formed by the junction of $\Gamma_1$ and $\Gamma_3$, Young's law is satisfied with respect to the prescribed normal $\hat{n}$. At the other vertices, the law is satisfied with respect to $\nu_1$ for the junction of $\Gamma_2$ and $\Gamma_3$; and with respect to $\nu_2$ for the vertex resulting from $\Gamma_1$ and $\Gamma_2$.

This configuration is uniquely determined up to translations, and we call the enclosed region the anisotropic Reuleaux triangle.

Finally, scaling the Wulff shape by a factor $\lambda>0$ scales the set
accordingly, and the enclosed area depends continuously and strictly
monotonically on $\lambda$. Therefore, there exists a unique $\lambda>0$ such
that the enclosed area equals the prescribed value $m$. This completes the
proof. 
\end{proof}

\end{lemma}

  \begin{figure}[htbp]
    \centering
    \includegraphics[width=0.55\textwidth]{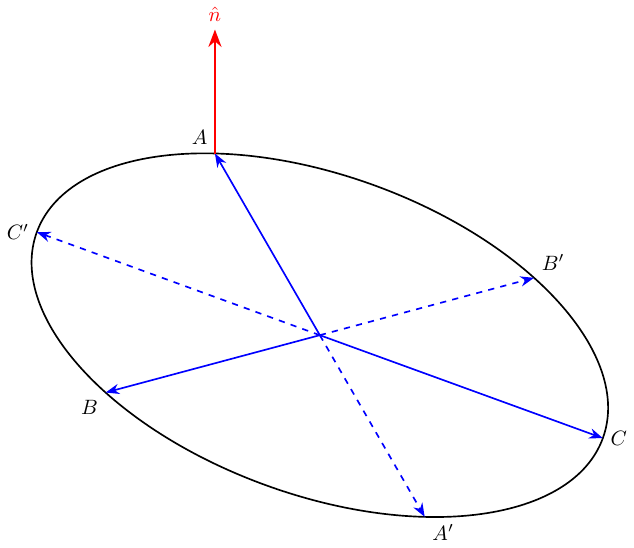} 
    \caption{Construction of the anisotropic Reuleaux triangle.}
    \label{fig:wulff_lens}
\end{figure}

\section{Regularity theory and local estimates} \label{section:regularity_th}
\begin{note}
    Most of these results hold in $\mathbb{R}^d$; however, for simplicity, we will state them for $\mathbb{R}^2$.
\end{note}

\begin{lemma}\cite[Corollary~5.8]{Novaga2023}\label{lemma:finitevolume_bounded} 
If $\mathcal{E}$ is a locally isoperimetric partition, then each region with finite measure is bounded. Hence, the subcluster of the chambers with finite volume is bounded and has finite perimeter.
    
\end{lemma}

\begin{theorem}\cite[Theorem~A]{Franceschi2023}\label{theorem:steiner} (Steiner regularity for minimal clusters). Let $\phi : \mathbb{S}^1 \to (0,+\infty)$  be of class $C^2$ and uniformly convex. Let $\mathcal{E}$ be a minimal cluster. Then $\mathcal{E}$ satisfies the Steiner property, i.e., $\partial \mathcal{E}$ is a locally finite union of $C^1$ arcs, and each junction point is the endpoint of exactly three different arcs, arriving with three different tangent vectors. The arcs of $\partial \mathcal{E}$ are actually $C^{1, \frac{1}{2}}$.

\end{theorem}

\begin{lemma}\label{lemma:upper_estim_perim} (Upper estimate of perimeter). Let $\phi$ be a general anisotropic density, and $\mathcal{E}=(E_1,\dots, E_S)$ be a locally isoperimetric $S$-partition in $\R^2$. Then, for every $\textbf{x}\in\R^2$ and any $R>0$, one has
\[
P_\phi(\mathcal{E}, B_R(\textbf{x})) \leq C_0 \cdot R^{1}
\]
with $C_0=C_0(\phi, 2, S)$ a constant not depending on $\mathcal{E}$, $\textbf{x}$ or $R$.

\begin{proof}
The argument is exactly the same as in Lemma~2.5 of \cite{Novaga2023}.
Up to translations assume $x=0$. For $0<\rho<R$ we apply the slicing
construction inside $B_\rho$, obtaining a partition $\mathcal F$ satisfying
(1)--(3). By local minimality, $P_\phi(\mathcal E,B_R)\le P_\phi(\mathcal F,B_R)$, thus
\[
2P_\phi(\mathcal E,B_R)
\le
2P_\phi(\mathcal E,B_R\setminus\overline{B}_\rho) + 2 \phi_{max} \mathcal{H}^{1}(\partial B_\rho)
+2\sum_{k=1}^S P_\phi(F_k,B_\rho)
\]
\[
\leq 2P_\phi(\mathcal E,B_R\setminus\overline{B}_\rho) + C_1\rho^{1}.
\]

Letting $\rho\to R^{-}$ and using
$P_\phi(\mathcal E,B_R\setminus\overline{B}_\rho)\to 0$
yields
\[
P_\phi(\mathcal E,B_R)\le C_0 R.
\]
Here, $C_0$ depends on $\phi$, $2$ and $S$.
\end{proof}

\end{lemma}

\begin{theorem}[Anisotropic closure for locally isoperimetric flat partitions]\label{theorem:closure} 
 Let $\phi : \mathbb{S}^1 \to (0,+\infty)$  be of class $C^2$ and uniformly convex. Let $\mathcal{E}^R = (E_1^R, \dots, E_S^R)$ be a sequence of locally isoperimetric $S$-partitions of $B_R \subset \mathbb{R}^2$ as $R \to +\infty$. 
Suppose there exists a partition $\mathcal{E}^\infty = (E_1^\infty, \dots, E_S^\infty)$ of $\mathbb{R}^2$ such that for all $j=1, \dots, S$, we have $E_j^R \to E_j^\infty$ in $L^1_{\text{loc}}(\mathbb{R}^2)$ as $R \to +\infty$.

If $\mathcal{E}^\infty$ is eventually flat, then $\mathcal{E}^\infty$ is a locally isoperimetric partition of $\R^2$. 

\begin{proof}
The proof follows by adapting the compactness-semicontinuity argument in \cite[Theorems~2.13 and 2.17]{Novaga2023} to the anisotropic setting. The key steps are the lower semicontinuity of the anisotropic perimeter and local density estimates.

\end{proof}
\end{theorem}

\begin{theorem} \label{theorem:boundary_cond}
 Let $\phi : \mathbb{S}^1 \to (0,+\infty)$  be of class $C^2$ and uniformly convex. Let $\mathcal{E}=(E_1,\dots, E_S)$ be a locally isoperimetric $S$-partition of the plane $\R^2$. Then $\partial E$ is connected. If two areas are infinite then $\partial\mathcal{E}$ coincides with two parallel rays outside a sufficiently large ball. If three areas are infinite then $\partial\mathcal{E}$ coincides with three half-lines whose prolongations define angles satisfying Young's law.
\begin{proof}
The proof is similar to the one of Theorem 4.2 in \cite{Novaga2023}.

By Lemma~\ref{lemma:finitevolume_bounded}, we get that regions with finite area are bounded. By Theorem~\ref{theorem:steiner} we know that the boundary of the cluster is a locally finite planar graph. Some arcs of the graph are unbounded, and in that case we assume each arc has one or two vertices of order $1$ at infinity. All other vertices, by the Steiner regularity theorem (Theorem~\ref{theorem:steiner}), have order three. We can find a radius $R>0$ large enough such that all bounded regions are compactly contained in $B_R$. Outside this ball, arcs of the graph have zero curvature, as there is no volume constraint on the infinite regions. So we characterize the graph outside $B_R$ by entire straight lines (with two end-points at infinity), line segments (with two endpoints in $\R^d$), and half-lines (with one end-point at $\R^d$ and one end-point at infinity).

We first claim that every bounded closed loop of $\partial \mathcal{E}$ is contained in $B_R$. Assume that a closed loop $\alpha$ contains an arc outside $B_R$; this would mean it is separating two infinite-area regions. Removing $\alpha$ and reassigning the enclosed area of one infinite component to another infinite region strictly decreases the perimeter, while preserving the volume constraints. This contradicts the minimality of $\mathcal{E}$, so we conclude that all closed loops are contained in $B_R$.

Moreover, we claim the graph $\partial \mathcal{E}$ has finitely many vertices. By Lemma~\ref{lemma:upper_estim_perim}, the number of vertices at infinity is bounded. Suppose we have an infinite number of vertices of order three. Since the graph is locally finite, these cannot accumulate in a compact region; so there must be a sequence of vertices diverging to infinity. As previously established, all closed loops of $\partial \mathcal{E}$ are contained in $B_R$. This number of loops is locally finite. By removing a finite number of arcs within $B_R$, we can break all cycles to obtain a cycle-free subgraph $\Gamma$, which consists of a finite number of trees intersecting $B_R$. One of the trees must be infinite, as $\Gamma$ contains a finite number of trees but an infinite number of vertices. So these vertices of order three must go to infinity. This guarantees the existence of infinitely many disjoint paths going to infinity, contributing to the linear perimeter and contradicting our density estimate of Lemma~\ref{lemma:upper_estim_perim}.

Since the graph $\partial \mathcal{E}$ is finite, by enlarging $R$ we can assume $B_R$ contains all vertices of order three, so that $\mathcal{E}\setminus B_R$ is composed of a finite number of disjoint half-lines, say $n$, emanating from $B_R$ and going to infinity. These half-lines have different directions towards infinity, as if we had two parallel half-lines we could merge them into a single long segment reducing the total perimeter. 

Recall that the interfaces separating the infinite-area regions have no volume constraints, meaning their associated Lagrange multipliers vanish. As established in \cite[Proposition 2.4]{Morgan1998}, the zero pressure difference implies that these interfaces must be straight half-lines outside $B_R$.

Now, we consider a \emph{blow-down} of $\mathcal{E}$. By letting $\mathcal{E}^k=\frac{\mathcal{E}}{k}$ be a rescaled partition of $\mathcal{E}$ we notice that $\mathcal{E}^k$ converges in $L_{\text{loc}}^1$ as $k\rightarrow \infty$ to a limit partition $\mathcal{E}^{\infty}$ of $\R^2$ delimited uniquely by $n$ half-lines, with a common vertex at the origin, each line parallel to the $n$-half-lines of $\mathcal{E}\setminus B_R$. Using \cite[Theorem 2.13]{Novaga2023}, adapted to the anisotropic case (the proof is the same using anisotropic density estimates), we deduce that $\mathcal{E}^\infty$ is a locally $J$-isoperimetric partition. The rescaled bounded regions, $E_j^k=\frac{E_j}{k}$, vanish in measure since their original areas $|E_j|$ are finite. So we can ignore the regions $E_j^\infty$ from $\mathcal{E}^{\infty}$, obtaining a locally isoperimetric partition of $\R^2$ composed of $n$ angles with a common vertex at the origin. By Theorem~\ref{theorem:steiner}, we know this junction must be a valid minimal cone satisfying Young's law. 

In the case $n=0$
 the graph has no vertices at infinity and the partition $\mathcal{E}$ is a bounded cluster. 
 In the case $n=3$ we have that $\partial \mathcal{E} \setminus \overline{B_R}$ is made of three half-lines going to infinity with angles satisfying Young's law. In the case of $n=2$ we have that $\partial \mathcal{E} \setminus \overline{B_R}$ is made of two half-lines going to infinity at opposite directions.

\end{proof}

\end{theorem}

\section{The minimization problem in $B_R$} \label{section:local_pb}

The asymptotic behavior of anisotropic perimeter minimizing clusters with two and three infinite areas is characterized by Theorem~\ref{theorem:boundary_cond}: outside a sufficiently large ball $B_R$, if two chambers have infinite area the cluster consists of two unbounded chambers separated by two parallel rays; if three chambers have infinite area the cluster consists of three unbounded chambers separated by three half lines. Therefore, our analysis focuses on the existence and uniqueness of the minimizer inside $B_R$ with prescribed boundary conditions.

\subsection{The $(1,2)$-cluster}

We consider a partition 
\[
\mathcal{E} = (E_1, E_2, E_3)
\]
of $\R^2$ satisfying:
\[
|E_1| = m > 0,\qquad |E_2| = |E_3| = \infty,
\]
for a fixed $m, R\in \R^+$, $R$ sufficiently large, in $\R^2$.

We want to minimize the anisotropic perimeter

\[
P(\mathcal{E}; B_R)
= \frac12 \sum_{i=1}^3 P(E_i; B_R),    \qquad
 B_R \subset \R^2.
\]

\begin{definition}[Exterior Boundary Condition]\label{def:exterior_condition}
We say that a partition $\mathcal{E}=(E_1,E_2,E_3)$ of $\mathbb{R}^2$ satisfies the
\emph{prescribed exterior configuration} outside the ball $B_R$ if:
\begin{enumerate}
    \item The finite phase is contained in the ball, i.e.,\ $E_1 \subset B_R$.
    \item There exists a fixed oriented line $L \subset \mathbb{R}^2$, with given normal $\hat{n}$, such that,
    outside the ball $B_R$, the phases $E_2$ and $E_3$ coincide with the two open
    half-planes determined by $L$. More precisely,
    \[
        \mathcal{E}\setminus B_R = \mathcal{E}_0 \setminus B_R,
    \]
    where $\mathcal{E}_0=(\emptyset,H^+,H^-)$ and $H^\pm$ are the two half-planes
    separated by $L$.
    \item We define $ \{p_1, p_2\}=L \cap \partial B_R$.
\end{enumerate}
\end{definition}

\begin{remark}[Reduction to a flat exterior configuration]\label{remark:flat_ext_cond}
The assumption that the exterior phases $E_2$ and $E_3$ are separated by a single
straight line $L$ does not entail any loss of generality.
The variational problem is localized inside $B_R$: admissible competitors
are required to coincide with the prescribed configuration outside the ball, and
all perimeter variations are compactly supported. 

If the minimizing configuration inside $B_R$ does not align with the exterior
direction, one may measure the resulting global misalignment by an angle
$\alpha(R)$ between the prescribed line $L$ and a reference direction
associated with the configuration inside the ball. 

It is then natural to expect that $\alpha(R)\to 0$ as $R\to\infty$, since a
non-vanishing misalignment would produce an excess anisotropic
perimeter at large scales, contradicting the minimality of the configuration. 

One could also measure the gap between the triple junctions of the lens construction in Lemma~\ref{lemma:anisotrop_lens} (determined by $\phi$ and $m$), and work with an exterior boundary condition adapted to this gap, namely fixing two half-lines separated by this distance. In that setting, the tilt parameter is avoided.
\end{remark}


\begin{proposition} [Local minimality of the Anisotropic Lens in $B_R$]\label{prop:local_lens}
Let $\phi$ be a regular anisotropy on $\mathbb{R}^2$. 
Given a radius $R>0$ sufficiently large, a direction $\hat{n}$, and a mass $m>0$, let $L_\phi$ be the associated \emph{anisotropic lens} from Lemma~\ref{lemma:anisotrop_lens}. We define the associated \emph{standard lens cluster} by joining the vertices of $L_\phi$ with two half-lines with prescribed normal $\hat{n}$ ending at the boundary of $B_R$. Then, this configuration is the unique minimizer of the anisotropic perimeter in $B_R$ among all clusters $\mathcal{E}$ satisfying the prescribed boundary condition (Definition~\ref{def:exterior_condition}) and the volume constraint. 
\end{proposition}

The proof of Proposition~\ref{prop:local_lens} is divided into two parts. First, we characterize the geometric properties of any minimizer; then, we establish its uniqueness. 

\subsection*{Existence in $B_R$}

Fix $R>0$. We consider the anisotropic perimeter minimization problem in $B_R$ with prescribed exterior configuration. By the direct method of the calculus of variations, there exists at least one minimizer. We aim to characterize it.

\begin{claim} \label{claim:claim1}
    $E(2,3)$ consists of straight line segments.

    \upshape{As there is no mass constraint for $E_2$ and $E_3$, there is no Lagrange multiplier in the Euler-Lagrange equations. As a consequence, minimizers must have zero curvature at points of $E(2,3)$.}
\end{claim}

\begin{claim}
    $E(1, 2)$ and $E(1,3)$ are arcs of constant anisotropic curvature of the same magnitude but opposite sign. 
    
     \upshape{This is a consequence of the Euler-Lagrange equations and Theorem~\ref{theorem:steiner}}. 
\end{claim}
\begin{claim}\label{claim:terminate_triple}
    Singular points of $E(i,j)$ occur when analytic arcs meet, and these can only occur at triple junction points.

    \upshape{This follows from Theorem \ref{theorem:steiner}. As a result, each connected arc of the transition $E(2,3)$ either meets a triple junction point, necessarily involving $E_1$, or extends to the boundary of the ball $B_R$. }
\end{claim}

\begin{claim} \label{claim:no_islands}
  $E_2$ and $E_3$ cannot have bounded connected components (no islands).

\end{claim}
  \begin{proof}
We assume, without loss of generality, that $E_2$ has a bounded connected component $\Omega \subset E_2$. If the boundary of this component, $\partial \Omega$, has non-trivial intersection with the boundary of $E_3$, $\partial E_3$, then we define a modified competitor cluster $\tilde{\mathcal{E}}=(\tilde{E}_1, \tilde{E}_2, \tilde{E}_3)$, which is identical to the original cluster $\mathcal{E}$ but we have reassigned the bounded connected component $\Omega \subset E_2$ to the region $E_3$. This is:

\[
\tilde{E}_1=E_1, \qquad
 \tilde{E}_2= E_2 \setminus \Omega, \qquad 
\tilde{E}_3=E_3 \cup \Omega .
\]

 $\tilde{\mathcal{E}}$ is an admissible cluster that agrees with $\mathcal{E}$ outside a compact set. However, its perimeter is strictly smaller, as we have erased an interface. This contradicts the minimality of $\mathcal{E}$.

If the boundary of $\Omega$ is disjoint from $E_3$, then it is necessarily completely surrounded by a connected component $\tilde{\Omega}$ of the region $E_1$. Now, we use a translational argument to create a partition with the same volume constraints and perimeter, but forcing $\Omega$ to meet the boundary of $\tilde{\Omega}$, violating the regularity condition. This also contradicts the minimizing property of $\mathcal{E}$.

We conclude neither $E_2$ nor $E_3$ can have bounded connected components. 
\end{proof}

\begin{claim}\label{claim:e1surrounded}
$E_1$ cannot have bounded connected components entirely surrounded by $E_2$ or $E_3$.
\end{claim}

\begin{proof}
Assume, without loss of generality, that $E_1$ has a bounded connected component 
$\Omega\subset E_1$ such that it is completely surrounded by $E_2$. We then create an admissible competitor by applying a translation and forcing $\Omega$ to touch the boundary of $E_2$. The new translated configuration is valid when it comes to preserving area and perimeter, but violates the regularity condition (no valid junction). Thus, we arrive at a contradiction with the minimality of $\mathcal{E}$. We deduce connected components of $E_1$ cannot be completely enclosed inside a component of an infinite phase chamber. 
\end{proof}

\begin{claim} \label{claim:two_tripl_junct}
    Each connected component of $E_1$ has exactly two triple junction points on its boundary.

    \begin{proof}
    First, every connected component of $E_1$ is bounded, since $E_1$ has finite volume (Lemma~\ref{lemma:finitevolume_bounded}). Moreover, each connected component of $E_1$ must contact both $E_2$ and $E_3$ on its boundary (Claim~\ref{claim:e1surrounded}). In addition, the number of connected components of $E_1$ is finite (Theorem~\ref{theorem:steiner}).

    Let $\Omega \subset E_1$ be  a connected component of $E_1$. Its outer boundary $\partial \Omega$ is a closed $C^1$ curve, and every point must belong either to the transition $E(1,2)$ or $E(1,3)$. Since $\partial \Omega $ must contain arcs of both types, the curve must switch from the interface $E(1,2)$ to $E(1,3)$ at some point. Claim~\ref{claim:terminate_triple} states that a change of phase can only happen at a triple junction. As $\partial \Omega$ is necessarily a closed curve, for each transition from $E(1,2)$ to $E(1,3)$  there must be a corresponding transition from $E(1,3)$ to $E(1,2)$ (with opposite orientation) to close the loop. Therefore, $\partial \Omega$ must contain $2n$ triple junctions, with $n\geq 1$.
    
    We argue by contradiction to show that $n=1$, that is the number of triple junctions being $2$. Pick one of the triple junctions of $\partial \Omega$ and call it $a_1$. From $a_1$ there emanate three curves: two of them belong to $\partial \Omega$; the other is an interface of the type $E(2,3)$, thus it lies entirely in $\overline {B_R} \setminus\Omega$.  We follow this $E(2,3)$-curve. First, by Steiner regularity (Theorem~\ref{theorem:steiner}) each interface is a $\mathcal{C}^1$ embedded arc and triple junctions are isolated. Therefore, this curve starting at $a_1$ cannot accumulate inside $B_R$. Second, the $E(2,3)$-curve cannot close up to form a bounded loop: such a loop would enclose a bounded connected component of either $E_2$ or $E_3$, contradicting Claim~\ref{claim:no_islands}. Thus, the curve cannot meet itself or another already-traced $E(2,3)$-curve in a way that produces a closed bounded region for the same reason. Third, because there are only finitely many connected components of $E_1$ (this follows from Theorem~\ref{theorem:steiner}), the curve cannot keep visiting new components $\Lambda_i \subset E_1$ forever; it must eventually terminate. By Theorem~\ref{theorem:steiner} the only possible terminations for an interior $E(2,3)$-arc are another triple junction (involving $E_1$) or the outer boundary of the partition,  $\partial B_R$. Considering previous exclusions, the $E(2,3)$-curve originated at $a_1$ cannot end at another triple junction without creating a forbidden bounded island, so it must eventually reach $\partial B_R$. 
    Now, turning to the exterior boundary condition: outside $B_R$ the partition coincides with a line intersecting $\partial B_R$ in two fixed points $p_1, p_2 \in \partial B_R$ . Hence, the interface $E(2,3)$ meets $\partial B_R$ only at the two points $p_1, p_2$. In particular, any curve $ \gamma \subset E(2,3)$ inside $\overline{B_R}$ that reaches the boundary must end at either $p_1$ or $p_2$. 

    We repeat the same construction for a second triple junction $a_2 \in\partial \Omega$. Its $E(2,3)$-curve likewise must end at $\partial B_R$, and it cannot meet the previously followed path before reaching the boundary, as it would produce a bounded region, contradicting Claim~\ref{claim:no_islands}. Therefore the two paths from $a_1$ and $a_2$ must end at two distinct points among $\{p_1, p_2 \}$. 
    If $n > 1$ (so $\partial \Omega$ has at least four triple junctions), pick a third triple junction $a_3$. The same reasoning enforces its $E(2,3)$-curve to end at $\partial B_R$, but both points $p_1$ and $p_2$ are occupied by disjoint paths from $a_1$ and $a_2$, and any merging prior to the boundary would create a contradiction with Claim~\ref{claim:no_islands}. This shows that $n>1$ is impossible.

   We deduce that necessarily $n=1$, and for each bounded connected component $\Omega \subset E_1$ there are exactly two triple junctions.

    \end{proof}
\end{claim}
\begin{claim}
    $E_1$ consists of a single anisotropic lens of total mass $m$.
\end{claim}

\begin{proof}

Let $\Omega_1, \dots , \Omega_r$ be the connected components of $E_1$. We analyze the global structure of these components using an auxiliary minimization problem. 
\medskip

\textit{Step 1: Reduction to a single component.}
 We first prove that, if $R$ is large enough, the bounded connected components $\{\Omega_i\}$ do not touch the boundary of $B_R$. We consider an auxiliary minimization problem, where we allow any possible translation of the components in the whole plane and any possible union of straight line segments such that the union of both sets is a compact connected set containing the two points $p_1$ and $p_2$. This is a problem that has a solution, since it only requires fixing the position of a finite number of segments, and the total length of these segments is continuous and coercive with respect to the position of the endpoints. So in this setting, $\Omega_i$ are not necessarily contained in $B_R$, and we aim to see that for $R$ large enough the original partition $\mathcal{E}$ is also a minimizer for this auxiliary problem. 

 Any minimizer of the auxiliary problem can be described as an abstract graph $\Gamma$ where the vertices are the components $\{\Omega_i\}$ and the points $\{p_1, p_2\}$, while the edges are the $E(2,3)$ interfaces. By Claim~\ref{claim:two_tripl_junct} and by considering that the points $p_1$ and $p_2$ are the only vertices of degree one, it follows that the graph is a simple path. Therefore, the components $\Omega_i$ form a single chain connecting $p_1$ and $p_2$. By the local description of triple junctions and Young's law, from each triple point there emanates a ray of the interface $E(2,3)$ which is parallel to all the other rays. Thus, all components are connected by parallel segments.

Now, we assume by contradiction that $E_1$ contains more than one connected component. Since all components lie along the same chain, any component may be translated in the direction parallel to these rays (i.e., perpendicular to the prescribed normal $\hat{n}$) without changing its volume or the total anisotropic perimeter until it comes into contact with another component. This modified configuration preserves the total perimeter and masses. However, the existence of such a point of contact contradicts the regularity theory for anisotropic triple junctions (Theorem~\ref{theorem:steiner}). Therefore, $E_1$ cannot have more than one connected component.

\medskip

\textit{Step 2: Equi-boundedness.} 
Since $E_1$ is formed by a single component $\Omega$ connected to $p_1$ and $p_2$ by parallel segments, and these points are diametrically opposite on $\partial B_R$, we can translate $\Omega$ along the axis parallel to the segments to centre it at the origin. For $R$ large enough, $\Omega$ is contained inside $B_R$. By Lemma~\ref{lemma:anisotrop_lens}, $\Omega$ must be an anisotropic lens of mass $m$ in order to satisfy all the regularity conditions. As we can translate the lens so it does not touch the boundary of $B_R$, this confirms that the original partition is the minimizer. 
\end{proof}

\begin{figure}[htbp]
    \centering
    \includegraphics[width=0.7\textwidth]{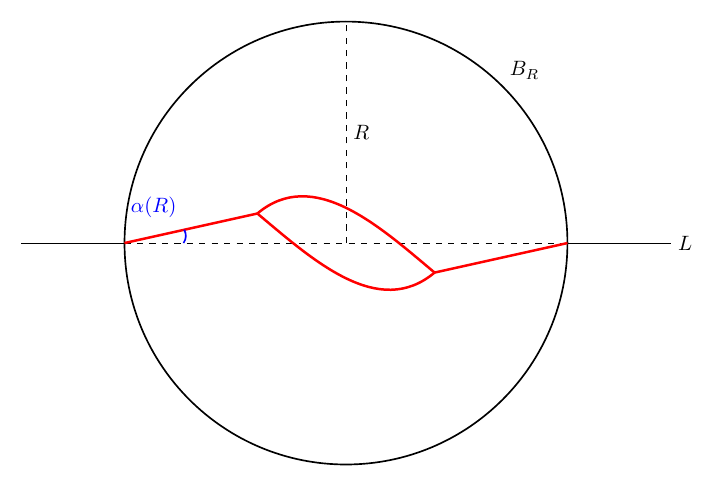} 
    \caption{Configuration of the minimal $(1,2)$-cluster inside $B_R$ for a given $R$.}
    \label{fig:tilt_lens}
\end{figure}

\subsection*{Uniqueness in $B_R$}

The previous steps already completely determine the structure of any minimizer. We now show that this structure is unique.

The local description implies that any minimizer must consist of a single interface attached to two exterior rays meeting $\partial B_R$ at the points $p_1,p_2$. By Claim~\ref{claim:no_islands} and Claim~\ref{claim:e1surrounded}, no additional components or alternative interfaces can exist inside $B_R$. Lemma~\ref{lemma:anisotrop_lens} shows that the anisotropic lens with fixed normal $\hat{n}$ and mass $m$, determined by the prescribed exterior configuration, is unique. Since no other configuration is compatible with the boundary conditions and the local optimality conditions, the full cluster is uniquely determined (up to translations in the direction of the exterior rays).

\begin{corollary}[Existence of the lens minimizers for general anisotropies] \label{cor:exist_gen_lens}
Let $\phi$ be a general anisotropy in $\R^2$. There exists a locally minimizing $(1,2)$-cluster with an anisotropic lens shape, satisfying the prescribed boundary condition (Definition~\ref{def:exterior_condition}) and the volume constraint.

\begin{proof}
    The proof follows by an approximation argument. Given any general anisotropy $\phi$ we can uniformly approximate it on $B_R$ by a sequence of $\mathcal{C}^2$ and uniformly convex anisotropies $\{\phi_k\}_{k\in \mathbb{N}}$. This uniform convergence implies the Hausdorff convergence of their corresponding Wulff shapes, $W_{\phi_k} \xrightarrow{\mathcal{H}} W_\phi$.

    By our previous topological and regularity results, which depend only on the smoothness and strict convexity of the anisotropy (from Claim~\ref{claim:claim1} to Claim~\ref{claim:two_tripl_junct}), for each $k\in \mathbb{N}$ there exists a minimizing cluster $\mathcal{E}_k$ which takes the form of an anisotropic lens (not necessarily symmetric). Specifically, its bounded chamber $E_{1,k}$ is convex and formed exactly by two arcs of scaled and translated Wulff shapes of $\phi_k$, connected at two triple junctions from which two exterior rays emanate. 
    
    The $\phi_k$-perimeters of these clusters are uniformly bounded (Lemma~\ref{lemma:upper_estim_perim}), and by the compactness theorem for sets of finite perimeter we deduce that, up to a subsequence, $\mathcal{E}_k$ converges in $L^1_{\text{loc}}$ to a limit cluster $\mathcal{E}$, with locally bounded perimeter, satisfying the constraints (the exterior condition assures no escape to infinity). 

    Using the uniform convergence $\phi_k \rightarrow \phi$ and the lower semicontinuity of the anisotropic perimeter with respect to $L^1_{\text{loc}}$ convergence, we obtain
    \[
    P_\phi(\mathcal{E}; B_R) \leq \liminf_{k \rightarrow \infty}{P_{\phi_k}(\mathcal{E}_k; B_R) }.
    \]
    Let $\mathcal{F}$ be any valid competitor cluster in $B_R$. We have $P_{\phi_k}(\mathcal{E}_k) \leq P_{\phi_k}(\mathcal{F}) + \Omega_k$, where $\Omega_k$ is an error term that  accounts for the difference between $\mathcal{E}_k$ and $\mathcal{E}$ near the boundary $\partial B_R$. Since $\mathcal{E}_k \xrightarrow{L^1_{\text{loc}}} \mathcal{E}$, this error term goes to zero as $k\rightarrow \infty$; and passing to the limit, using the uniform convergence and the lower semicontinuity, we obtain $P_\phi(\mathcal{E}) \leq P_\phi(\mathcal{F})$.

    Finally, since the bounded chambers $E_{1,k}$ are convex and the perimeter (and area) are locally bounded, by Blaschke's selection theorem we have that $\partial E_{1,k} \xrightarrow{\mathcal{H}} \partial E_1$, with $E_1$ convex and with bounded perimeter. Because the boundaries of the bounded chambers $E_{1,k}$ are composed of arcs of $W_{\phi_k}$, the limit boundary of the finite chamber $E_1$ must be composed of arcs of the limit Wulff shape $W_\phi$. The triple junctions of $E_{1,k}$ converge to triple junctions in $E_1$, and the interfaces $E(2,3)_k$ naturally converge to the corresponding limit rays. 
    
    A convex lens made by Wulff arcs, together with the rays separating $E_2$ from $E_3$, constitutes the anisotropic lens cluster, which concludes the proof.
    
\end{proof}
    
\end{corollary}

\begin{remark}\label{rem:lost_prop}
    The regularity and uniqueness of the minimizer are generally lost in the limit process. The boundary of the bounded chamber $E_{1}$ consists of two arcs which are limits of smooth Wulff arcs, but the strict $\mathcal{C}^1$ regularity of the interfaces might be lost in the general $W_\phi$ limit. Moreover, other minimal configurations may exist, featuring non-standard singularities such as quadruple junctions.
\end{remark}

\begin{example}[Loss of uniqueness with a general anisotropy]
As an example of a general anisotropy in $\R^2$, let us consider the positively 1-homogeneous convex function defined by  
\[
\phi(\nu)= |\nu_1|+|\nu_2|.
\]
Its associated Wulff shape is the square $[-1,1]^2$.

Suppose we want to find a local minimizer for a $(1,2)$-cluster under the anisotropic density. Since quadruple junctions can be stable for this anisotropy, we expect the configuration illustrated in Figure~\ref{fig:square} to be a valid minimizer under certain exterior boundary conditions. The existence of this configuration strongly suggests the non-uniqueness of the lens type minimizer.

\begin{figure}[htbp]
    \centering
    \includegraphics[width=0.5\textwidth]{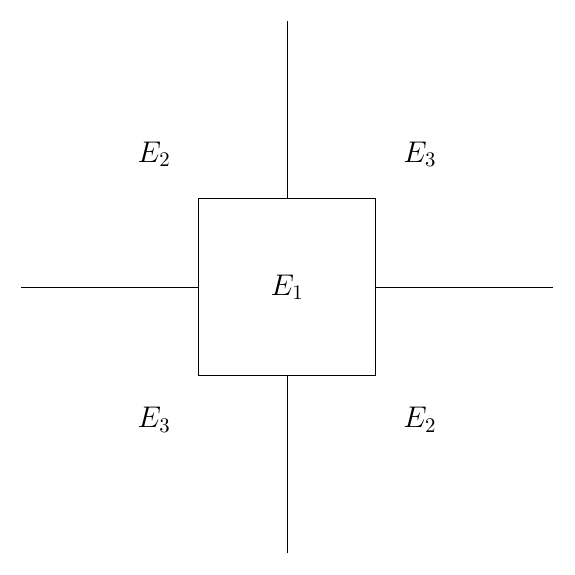} 
    \caption{Example of a possible minimizer with a non lens shape, given by the anisotropic density $\phi(\nu)= |\nu_1|+|\nu_2|.$}
    \label{fig:square}
\end{figure}
\end{example}

\subsection{The $(1,3)$-cluster}

We consider a partition 
\[
\mathcal{E} = (E_1, E_2, E_3, E_4)
\]
of $\R^2$ satisfying:
\[
|E_1| = m > 0,\qquad |E_2| = |E_3| = |E_4| =\infty,
\]
for a fixed $m, R\in \R^+$, $R$ sufficiently large, in $\R^2$.

We want to minimize the anisotropic perimeter

\[
P(\mathcal{E}; B_R)
= \frac12 \sum_{i=1}^4 P(E_i; B_R),    \qquad
 B_R \subset \R^2.
\]

\begin{definition}[Exterior boundary condition]\label{def:exterior_condition_triple}
We say that a partition $\mathcal{E}=(E_1,E_2,E_3,E_4)$ of $\R^2$ satisfies the
\emph{prescribed exterior configuration} outside the ball $B_R$ if:
\begin{enumerate}
    \item The finite phase is contained in the ball, i.e.,\ $E_1 \subset B_R$.

    \item There exist three distinct half-lines $L_1,L_2,L_3\subset\R^2$ starting at the origin such that,
    outside the ball $B_R$, the boundary of the partition coincides with their union:
    \[
        \partial\mathcal{E}\setminus B_R
        = (L_1\cup L_2\cup L_3)\setminus B_R .
    \]

    \item If $\nu_1,\nu_2,\nu_3\in\mathbb{S}^1$ denote the unit normals to the interfaces
    between the infinite phases determined by the half-lines $L_i$, then the directions of
    $L_1,L_2,L_3$ satisfy the anisotropic Young law
    \[
        \nabla\phi(\nu_1)+\nabla\phi(\nu_2)+\nabla\phi(\nu_3)=0 .
    \]
     \item We define $ \{p_1, p_2, p_3\}=(L_1\cup L_2\cup L_3) \cap \partial B_R$.
\end{enumerate}
\end{definition}

\begin{remark}[Reduction to a fixed exterior configuration]\label{remark:flat_ext_cond3}
    The assumption of the lines separating the exterior phases $E_2$, $E_3$, and $E_4$ being fixed and satisfying Young's Law entails no loss of generality, analogous to the lens case (Remark~\ref{remark:flat_ext_cond}).
\end{remark}


\begin{proposition}[Local minimality of the Anisotropic triod n $B_R$]\label{prop:local_triod}
Let $\phi$ be a regular anisotropy on $\R^2$. 
Given a radius $R>0$ sufficiently large, a direction $\hat{n}$, and a mass $m>0$, let $T_\phi$ be the associated \emph{anisotropic Reuleaux triangle} from Lemma~\ref{lemma:anisotrop_triple_lens}. We define the associated \emph{standard anisotropic triod cluster} by joining the vertices of $T_\phi$ with three half-lines with the normals required to fulfill Young's law, ending at the boundary of $B_R$. Then, this configuration is the unique minimizer of the anisotropic perimeter in $B_R$ among all clusters $\mathcal{E}$ satisfying the prescribed boundary condition (Definition~\ref{def:exterior_condition_triple}) and the volume constraint. 
\end{proposition}

The proof of Proposition~\ref{prop:local_triod} is divided into two parts. First, we characterize the geometric properties of any minimizer; then, we prove its uniqueness. 

\subsection*{Existence in $B_R$}

Fix $R>0$. We consider the anisotropic perimeter minimization problem in $B_R$ with prescribed exterior configuration. By the direct method of the calculus of variations, there exists at least one minimizer. We aim to characterize it.

\begin{claim}  \label{claim:straight_arcs_triple}
    $E(2,3)$, $E(2,4)$, $E(3,4)$ are straight line segments.

    \upshape{As there is no mass constraint for $E_2$, $E_3$ and $E_4$, there is no Lagrange multiplier in the Euler-Lagrange equation. As a consequence, minimizers must have zero curvature at points of $E(2,3)$, $E(2,4)$, $E(3,4)$.}
\end{claim}

\begin{claim}
    $E(1, 2)$, $E(1,3)$, $E(1,4)$ are arcs of constant anisotropic curvature of the same magnitude. 
    
     \upshape{This is a consequence of the Euler-Lagrange equations and Theorem~\ref{theorem:steiner}}. 
\end{claim}
\begin{claim}\label{claim:terminate_triple2}
    Singular points of $E(i,j)$ occur when analytic arcs meet, and these can only occur at triple junction points.

    \upshape{This follows from Theorem \ref{theorem:steiner}. As a result, each connected arc of the transition $E(i,j)$ either meets a triple junction point, or extends to the boundary of the ball $B_R$. }
\end{claim}

\begin{claim}  \label{claim:no_islands_triple}
    $E_2$, $E_3$ and $E_4$ cannot have bounded connected components (no islands).
    \end{claim}

    \begin{proof}
Assume that $E_2$ (analogously for $E_3$ or $E_4$) has a bounded connected component $\Omega$. If $\partial \Omega$ has a non-trivial intersection with the boundary of $E_3$ or $E_4$, then we define a modified competitor cluster $\tilde{\mathcal{E}}=(\tilde{E}_1, \tilde{E}_2, \tilde{E}_3,\tilde{E}_4)$ by reassigning the bounded component to either region $E_3$ or $E_4$, depending on which region its boundary intersects. $\tilde{\mathcal{E}}$ is an admissible cluster that agrees with $\mathcal{E}$ outside a compact set. However, its perimeter is strictly smaller, as an interface has been removed. This contradicts the minimality of $\mathcal{E}$.

If the boundary of $\Omega$ is disjoint from both $E_3$ and $E_4$, then it is necessarily completely surrounded by a connected component $\tilde{\Omega}$ of the region $E_1$. Now, we use a translational argument to construct a partition with the same volume constraints and perimeter, but forcing $\Omega$ to meet the boundary of $\tilde{\Omega} \subset E_1$, violating the regularity condition. This also contradicts the minimizing property of $\mathcal{E}$.

We conclude that none of the phases $E_2$, $E_3$, or $E_4$ can have bounded connected components. 
\end{proof}

\begin{claim}\label{claim:e1surrounded_triple}
$E_1$ cannot have bounded connected components entirely surrounded by $E_2$, $E_3$ or $E_4$.
\end{claim}

\begin{proof}
Assume, without loss of generality, that $E_1$ has a bounded connected component 
$\Omega\subset E_1$ which is entirely surrounded by $E_2$. We construct an admissible competitor by translating $\Omega$ until it touches the boundary of $E_2$. This modified configuration remains valid with respect to the preservation of area and perimeter, but does not fulfill the regularity conditions, as it has a non-regular junction. Consequently, we arrive at a contradiction to the minimality of $\mathcal{E}$. We deduce that connected components of $E_1$ cannot be completely enclosed inside a component of an infinite phase chamber. 
\end{proof}

\begin{claim}\label{claim:unique_component_triple}
The set $E_1$ has exactly one connected component.
Moreover, this component is congruent to the anisotropic Reuleaux triangle associated with the prescribed exterior configuration.
\end{claim}

\begin{proof}
The proof proceeds in several steps. First, we show that $E_1$ has only one connected component. Then, we prove that it is necessarily congruent to the anisotropic triod described in Lemma~\ref{lemma:anisotrop_triple_lens}.
\medskip

    \textit{Step 1: Single connected component of $E_1$.} 
    We consider the connected components $\Omega_1, \Omega_2,\dots \subset E_1$. The arcs of $\partial E$ not contained in $E_1$ are separating two infinite regions, and by Claim~\ref{claim:straight_arcs_triple} are straight line segments. The union of the different connected components of $E_1$ together with these segments forms a compact connected set containing the three points $\{p_1, p_2, p_3\}$ (otherwise, the partition would not align with the boundary condition).

    We consider an auxiliary minimization problem where we allow any possible translation of each component in the plane and any possible configuration of straight line segments such that the union of the components and the segments forms a compact connected set containing the three fixed points $p_1, p_2, p_3$. We seek the configuration that minimizes the total anisotropic perimeter, subject to the volume constraints and the condition that different components do not overlap. 

By Lemma~\ref{lemma:finitevolume_bounded}, each component is uniformly bounded. The number of segments is also bounded, as connecting two components multiple times would create extra unnecessary perimeter. The existence of a minimizer is guaranteed because the problem involves a finite number of segments, and the total length functional is both continuous and coercive with respect to the segment endpoints. 

The topology of the minimizer is described by a graph $\Gamma$, where the edges correspond to the line segments. The vertices of $\Gamma$ represent the three points $p_1, p_2, p_3$, the components $\Omega_1, \dots, \Omega_k$, and potential junctions where three segments meet (satisfying the anisotropic Young's law). We can exclude configurations where two components touch or a component meets one of the points $p_i$, as such contacts would necessarily create junctions where more than three arcs meet, which are forbidden in minimizing clusters.

The graph $\Gamma$ is connected by construction. The fixed points $\{p_1, p_2, p_3\}$ are vertices of degree one. No other vertices can have degree one; otherwise, the attached component could be translated along the segment to strictly decrease the anisotropic energy. Furthermore, $\Gamma$ contains no cycles, as removing a segment from a cycle would reduce the total perimeter. Thus, $\Gamma$ is a tree.

Now, consider a vertex of $\Gamma$ with degree two. This represents a component $\Omega_i$ connected to two other nodes, say $A$ and $B$, via segments. The junctions between $\Omega_i$ and its attached segments must satisfy the anisotropic Young's Law. This forces the two segments to be parallel to each other. By translational invariance, the anisotropic perimeter is constant as we move the component along the direction of these segments. We can therefore translate the component $\Omega_i$ in this direction without increasing the perimeter until it comes into contact with another component (or one of the points $p_j$). This configuration would violate the standard regularity of minimizers. Thus, we conclude that the graph $\Gamma$ contains no vertices of degree two.

A tree with exactly three leaves (vertices of degree one) and no vertices of degree two must be a triod (a single internal node of degree three connected to the three leaves). This implies that the minimizer consists of a single component, $\Omega \subset E_1$, with three segments connecting it to $p_1, p_2, p_3$. 

Now, we aim to prove that for $R$ large enough the bounded component tends to remain near the origin, thus, $\Omega$ is uniformly bounded as $R \rightarrow \infty$. Inside the ball $B_R$, the total anisotropic perimeter is given by $P_\phi (\Omega)+\sum_{i=1}^3 \phi^*(p_i-q_i)$, where $q_i$ are the three vertices of the bounded component. Let $q$ denote the centre of the component. We study the minimization of $f(q)=\sum_{i=1}^3 \phi^*(p_i-q)$, since the perimeter of $\Omega$ is bounded (due to the mass constraint), for $R$ large enough we can neglect it, as well as the distances between the vertices $q_i$ and the centre $q$. Since the points $p_i$ were chosen to satisfy Young's law, we deduce that the origin $q=0$ minimizes the cost $f(q)$. Consequently, any translation of the component away from the origin would increase the total anisotropic perimeter. Thus, for sufficiently large $R$, the configuration with the component centred at the origin is the minimizer, ensuring that $\Omega$ does not touch the boundary of the ball. We conclude that the original minimizer is a valid minimizer for the auxiliary problem.

    \medskip

    \textit{Step 2: $E_1$ is congruent to the anisotropic Reuleaux triangle.}    
    By Lemma~\ref{lemma:anisotrop_triple_lens}, once an exterior normal $\hat n$ is fixed (as in the prescribed exterior configuration), there
exists a unique (up to translations) bounded set whose boundary
consists of three arcs of a scaled Wulff shape meeting at three triple junctions and
satisfying Young's law: the anisotropic Reuleaux triangle. The interfaces connecting the finite phase $E_1$ to $\partial B_R$ meet with a nonzero tilt. In this case, one can naturally define an
angle $\alpha_i(R)$ between the exterior reference lines $L_i$ and the directions
of these connecting interfaces for $i\in\{1, 2, 3\}$, as mentioned in Remark~\ref{remark:flat_ext_cond3}.    
\end{proof}

\begin{figure}[htbp]
    \centering
    \includegraphics[width=0.45\textwidth]{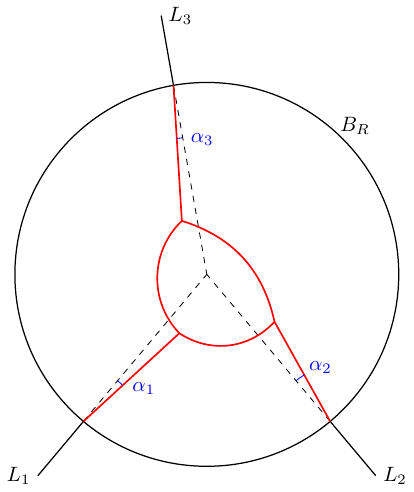} 
    \caption{Configuration of the minimal $(1,3)$-cluster inside $B_R$ for a given $R$.}
    \label{fig:tilt_triod}
\end{figure}

\subsection*{Uniqueness in $B_R$}

The previous steps already completely determine the structure of any minimizer. We now show that this structure is unique.

The local description implies that any minimizer must consist of a single interface attached to three external rays meeting $\partial B_R$ at the points $p_1$, $p_2$, and $p_3$. By Claim~\ref{claim:no_islands_triple} and Claim~\ref{claim:unique_component_triple}, no additional components or alternative interfaces can exist inside $B_R$. Lemma~\ref{lemma:anisotrop_triple_lens} shows that the set connecting these three points, with mass $m$ and the adequate tangent vectors, is unique. Since no other configuration is compatible with the boundary conditions and the local optimality conditions, the full cluster is uniquely determined (up to translations in the direction of the exterior rays).

\begin{corollary}[Existence of the triod minimizer for general anisotropies] \label{cor:exist_gen_triod}
Let $\phi$ be a general anisotropy in $\R^2$. There exists a locally minimizing $(1,3)$-cluster with an anisotropic triod shape, satisfying the prescribed boundary condition (Definition~\ref{def:exterior_condition_triple}) and the volume constraint.

\begin{proof}
    The proof follows exactly as in Corollary~\ref{cor:exist_gen_lens}.     
\end{proof}
    
\end{corollary}

\begin{remark}
    The regularity and uniqueness of the minimizer are generally lost in the limit process, as argued in Remark~\ref{rem:lost_prop}. 
\end{remark}

\section{Main Results} \label{section:main}

Now we aim to study the asymptotic behavior of the minimizers $\mathcal{E}_R$ as $R \to \infty$ for the $(1,2)$-cluster and the $(1,3)$-cluster.

\subsection{The $(1,2)$-cluster}

\begin{theorem} \label{th:ifonlyif_lens}
Let $\phi$ be a regular anisotropy and $m > 0$. A $(1,2)$-cluster $\mathcal{E}$ in $\R^2$, whose finite chamber has mass $m$, is a local minimizer of the anisotropic perimeter if and only if, up to translations, $\mathcal{E}$ is a \emph{standard anisotropic lens cluster} of mass $m$ oriented in a minimizing direction $\hat{n}$.
\begin{proof}
We will divide the proof into two steps, accounting for both implications.

($\Rightarrow$): 
Let $\mathcal{E}$ be a locally minimizing $(1,2)$-cluster in $\R^2$ for the anisotropic perimeter, whose finite chamber has mass $m$. By Theorem~\ref{theorem:boundary_cond}, we can characterize its asymptotic behavior by a normal direction $\hat{n}$. Given this direction $\hat{n}$, we fix a boundary condition as in Definition~\ref{def:exterior_condition} for any $R>0$.

Let $\mathcal{E}_R$ be a sequence of minimal $(1,2)$-clusters constructed in $B_R$ (Proposition~\ref{prop:local_lens} and Corollary~\ref{cor:exist_gen_lens}). Using Theorem~\ref{theorem:closure}, this sequence converges in $L^1_{\text{loc}}(\R^2)$, up to subsequences, to a locally isoperimetric limit partition $\mathcal{E}_\infty$ as $R \rightarrow \infty$. 

We first show that the tilt $\alpha(R)$ of the connecting interfaces must vanish as $R \to \infty$. A non-vanishing tilt $\alpha(R)$ as $R\to\infty$ would imply that their length inside $B_R$ is asymptotically comparable to $R/\cos\alpha(R)$. This would generate a linear excess of anisotropic perimeter with respect to a horizontal competitor aligned with $L$, contradicting minimality; hence,
$\alpha(R) \to 0$. 

Consequently, the limit partition $\mathcal{E}_\infty$ consists of a line $L$ with normal $\hat{n}$ separating the two infinite phases $E_2$ and $E_3$, together with an anisotropic lens.  By Lemma~\ref{lemma:anisotrop_lens} this configuration is uniquely determined. Since the local minimizer is unique under these regularity assumptions, our original partition $\mathcal{E}$ must coincide, up to translations, with $\mathcal{E}_\infty$, which concludes the characterization of $\mathcal{E}$ as the \emph{standard anisotropic lens cluster} of mass $m$ oriented in a minimizing direction $\hat{n}$. 

($\Leftarrow$): 
Conversely, let $\mathcal{L}_\phi$ be a \emph{standard anisotropic lens cluster} of mass $m$ oriented in a minimizing direction $\hat{n}$. From the implication above, we know that a local minimizer exists, up to translations, and it is given by the \emph{standard anisotropic lens cluster}. Let $P_{\phi, \text{min}}$ be the anisotropic perimeter of this minimizer in $B_R$, for a given $R>0$ large enough. Our given lens cluster $\mathcal{L}_\phi$ shares exactly the same perimeter inside $B_R$ as this minimizer. It follows that $\mathcal{L}_\phi$ achieves the minimum perimeter in any $B_R$, and therefore is a local minimizer. 

\end{proof}
    
\end{theorem}

\begin{corollary} \label{cor:mainresult_cor_lens}
Let $\phi$ be a general anisotropy and $m > 0$. Then there exists a direction $\hat{n}$ such that the $(1,2)$-cluster $\mathcal{E}$ in $\R^2$ given by the \emph{standard anisotropic lens cluster} of mass $m$ oriented in $\hat{n}$ is a local minimizer of the anisotropic perimeter.

\begin{proof}
The proof follows from the previous theorem and an approximation argument. Given any general anisotropy $\phi$ we can uniformly approximate it by a sequence of $\mathcal{C}^2$ and uniformly convex anisotropies $\{\phi_k\}_{k\in \mathbb{N}}$. For each regular anisotropy, let $\mathcal{L}_{\phi_k}$ be the \emph{standard anisotropic lens cluster} of mass $m$ oriented in a minimizing direction $\hat{n}_k$, corresponding to the asymptotic behavior of the minimizing cluster. This is a sequence of minimizers (Theorem~\ref{th:ifonlyif_lens}). We take the limit as $k\rightarrow\infty$, up to a subsequence, obtaining $\mathcal{L}_{\phi}$, which corresponds to the \emph{standard anisotropic lens cluster} of mass $m$ oriented in a minimizing direction $\hat{n}$. 

We aim to prove it is a local minimizer. Using the same argument as in the proof of Corollary~\ref{cor:exist_gen_lens}, for any valid competitor cluster $\mathcal{F}$, we obtain $P_\phi(\mathcal{L}_\phi) \leq P_\phi(\mathcal{F})$.

This proves that the \emph{standard anisotropic lens cluster} of mass $m$ oriented in $\hat{n}$ is a local minimizer of the anisotropic perimeter for any general anisotropy $\phi$.
\end{proof}
    
\end{corollary}

\subsection{The $(1,3)$-cluster}

\begin{theorem} \label{th:ifonlyif_triod}
Let $\phi$ be a regular anisotropy and $m > 0$. A $(1,3)$-cluster $\mathcal{E}$ in $\R^2$, whose finite chamber has mass $m$, is a local minimizer of the anisotropic perimeter if and only if, up to translations, $\mathcal{E}$ is a \emph{standard anisotropic triod cluster} of mass $m$ oriented in a minimizing direction $\hat{n}$.
\begin{proof}
We will divide the proof into two steps, accounting for both implications.

($\Rightarrow$): 
Let $\mathcal{E}$ be a locally minimizing $(1,3)$-cluster in $\R^2$ for the anisotropic perimeter, whose finite chamber has mass $m$. By Theorem~\ref{theorem:boundary_cond}, we can characterize its asymptotic behavior by a normal direction $\hat{n}$. Given this direction $\hat{n}$, we fix a boundary condition as in Definition~\ref{def:exterior_condition_triple} for any $R>0$.

Let $\mathcal{E}_R$ be a sequence of minimal $(1,3)$-clusters constructed in $B_R$ (Proposition~\ref{prop:local_triod} and Corollary~\ref{cor:exist_gen_triod}). As shown in Claim~\ref{claim:unique_component_triple}, for $R$ large enough we can affirm that the finite chamber $E_1^R$ remains uniformly bounded near the origin, thus contained in a fixed ball. Using Theorem~\ref{theorem:closure}, this sequence converges locally in $L^1(\R^2)$, up to subsequences, to a locally isoperimetric limit partition $\mathcal{E}_\infty$ as $R \rightarrow \infty$. 

Analogously to the lens case, the tilts  $\alpha_i(R)$ of the connecting interfaces must vanish as $R \to \infty$; we deduce it using the same argument of minimality as above. 

Consequently, the limit partition $\mathcal{E}_\infty$ consists of three rays $L_1$, $L_2$ and $L_3$, satisfying Young's law, separating the infinite phases $E_2$, $E_3$ and $E_4$, together with an anisotropic triod. By Lemma~\ref{lemma:anisotrop_triple_lens} this configuration is uniquely determined. Since the local minimizer is unique under these regularity assumptions, our original partition $\mathcal{E}$ must coincide, up to translations, with $\mathcal{E}_\infty$, which concludes the characterization of $\mathcal{E}$ as the \emph{standard anisotropic triod cluster} of mass $m$ oriented in a minimizing direction $\hat{n}$. 

($\Leftarrow$): 
Conversely, let $\mathcal{T}_\phi$ be a \emph{standard anisotropic triod cluster} of mass $m$ oriented in a minimizing direction $\hat{n}$. From the implication above, we know that a local minimizer exists, up to translations, and it is given by the \emph{standard anisotropic triod cluster}. Let $P_{\phi, \text{min}}$ be the anisotropic perimeter of this minimizer in $B_R$, for a given $R>0$ large enough. Our given triod cluster $\mathcal{T}_\phi$ shares exactly the same perimeter inside $B_R$ as this minimizer. It follows that $\mathcal{T}_\phi$ achieves the minimum perimeter in any $B_R$, and therefore is a local minimizer. 

\end{proof}
    
\end{theorem}

\begin{corollary} \label{cor:mainresult_cor_triod}
Let $\phi$ be a general anisotropy and $m > 0$. Then there exists a direction $\hat{n}$ such that the $(1,3)$-cluster $\mathcal{E}$ in $\R^2$ given by the \emph{standard anisotropic triod cluster} of mass $m$ oriented in $\hat{n}$ is a local minimizer of the anisotropic perimeter.

\begin{proof}
The proof follows from the previous theorem and an approximation argument. Given any general anisotropy $\phi$ we can uniformly approximate it by a sequence of $\mathcal{C}^2$ and uniformly convex anisotropies $\{\phi_k\}_{k\in \mathbb{N}}$. For each regular anisotropy, let $\mathcal{T}_{\phi_k}$ be the \emph{standard anisotropic triod cluster} of mass $m$ oriented in a minimizing direction $\hat{n}_k$, corresponding to the asymptotic behavior of the minimizing cluster. This is a sequence of minimizers (Theorem~\ref{th:ifonlyif_triod}). We take the limit as $k\rightarrow\infty$, up to a subsequence, obtaining $\mathcal{T}_{\phi}$, which corresponds to the \emph{standard anisotropic triod cluster} of mass $m$ oriented in a minimizing direction $\hat{n}$. 

We aim to prove it is a local minimizer. Using the same argument as in the proof of Corollary~\ref{cor:exist_gen_triod}, for any valid competitor cluster $\mathcal{F}$, we obtain $P_\phi(\mathcal{T}_\phi) \leq P_\phi(\mathcal{F})$.

This proves that the \emph{standard anisotropic triod cluster} of mass $m$ oriented in $\hat{n}$ is a local minimizer of the anisotropic perimeter for a general anisotropy $\phi$.
\end{proof}
\end{corollary}

\section{Acknowledgments}

I appreciate the support received from the UPC Interdisciplinary Higher Education Center (CFIS), which funded the research stay during which this work was conducted. I also express my sincere gratitude to Prof. Matteo Novaga for his constant guidance, and for suggesting the problem that led to this work. 

\printbibliography

\end{document}